\documentclass{amsart}
\usepackage{verbatim,amssymb,amscd,graphicx,psfrag}
\usepackage[all]{xy}
\numberwithin{equation}{section}
\let\cal\mathcal
\def\Ascr{{\cal A}}

\def\Cscr{{\cal C}}
\def\Dscr{{\cal D}}

\def\Hscr{{\cal H}}

\def\Oscr{{\cal O}}
\def\Pscr{{\cal P}}

\def\Sscr{{\cal S}}
\def\Tscr{{\cal T}}

\let\blb\mathbb

\def \PP{{\blb P}}

\def \ZZ{{\blb Z}}

\newcommand{\proj}{\operatorname{proj}}
\def\id{\text{id}}
\def\Id{\operatorname{id}}

\def\Lotimes{\overset{L}{\otimes}}

\def\Mod{\operatorname{Mod}}
\def\mod{\operatorname{mod}}
\def\Gr{\operatorname{Gr}}

\def\qgr{\operatorname{qgr}}

\def\gr{\operatorname{gr}}

\def\Qch{\operatorname{Qch}}
\def\coh{\mathop{\text{\upshape{coh}}}}

\def\gr{\operatorname {gr}}

\def\Ext{\operatorname {Ext}}
\def\Hom{\operatorname {Hom}}

\def\End{\operatorname {End}}
\def\RHom{\operatorname {RHom}}

\def\Sl{\operatorname {Sl}}

\def\ker{\operatorname {ker}}
\def\Ker{\operatorname {ker}}

\def\End{\operatorname {End}}

\def\id{{\operatorname {id}}}

\def\r{\rightarrow}

\DeclareMathOperator{\Proj}{Proj}

\DeclareMathOperator{\Ind}{Ind}

\let\dirlim\injlim

\newtheorem{lemma}{Lemma}[section]
\newtheorem{proposition}[lemma]{Proposition}
\newtheorem{theorem}[lemma]{Theorem}

\newtheorem{lemmas}{Lemma}[subsection]
\newtheorem{propositions}[lemmas]{Proposition}
\newtheorem{theorems}[lemmas]{Theorem}

\theoremstyle{definition}

\newtheorem{example}[lemma]{Example}

\newtheorem{examples}[lemmas]{Example}

{

}

\theoremstyle{remark}

\newtheorem{remark}[lemma]{Remark}
\newtheorem{remarks}[lemmas]{Remark}

\newdimen\uboxsep \uboxsep=1ex
\def\uboxn#1{\vtop to 0pt{\hrule height 0pt depth 0pt\vskip\uboxsep
\hbox to 0pt{\hss #1\hss}\vss}}

\def\uboxs#1{\vbox to 0pt{\vss\hbox to 0pt{\hss #1\hss}
\vskip\uboxsep\hrule height 0pt depth 0pt}}

\def\MCM{\operatorname{MCM}}
\def\uMCM{\underline{\operatorname{MCM}}}

\def\Inj{\operatorname{Inj}}

\def\can{\operatorname{can}}
\def\ac{\operatorname{ac}}
\def\K{K}
\def\sing{\operatorname{Sg}}
\def\straightK{\operatorname{K}}
\def\straightC{\operatorname{C}}

\def\dbsing{D_{\sing}}
\def\dbsinge{D'_{\sing}}
\newcommand{\koszul}[2]{\operatorname{E}[#1]}
\newcommand{\ukoszul}[2]{\straightK[#1]}
\newcommand{\cech}[2]{\check{\straightC}[#1]}
\newcommand{\cat}[1]{\mathcal{#1}}
\newcommand{\qder}[1]{D(#1)}

\newcommand{\qderu}[2]{D^{#1}(#2)}
\newcommand{\lto}{\longrightarrow}
\newcommand{\xlto}[1]{\stackrel{#1}\lto}
\newcommand{\mf}[1]{\mathfrak{#1}}

\author{Bernhard Keller}
\address{UFR de Mathématiques\\
Université Denis Diderot - Paris 7\\
2, place Jussieu\\
75251 Paris Cedex 05\\
FRANCE}
\email{keller@math.jussieu.fr}
\author{Daniel Murfet}
\address{Hausdorff Center for Mathematics, Endenicher Allee 62, D-53115 Bonn, Germany}
\email{murfet@math.uni-bonn.de}
\author{Michel Van den Bergh}
\address{Departement WNI\\Universiteit Hasselt\\ Universitaire Campus\\ Building D\\ 3590
Diepenbeek\\ Belgium}
\email{michel.vandbenbergh@uhasselt.be}
\thanks{The third author is a director of research at the FWO}
\thanks{The results of this paper were partially obtained during visits
of the third author to
the
Universit\'e Paris Diderot-Paris 7
and the Centre de Recerca Matem\`atica in Barcelona. He hereby thanks
these institutions for their
hospitality.}

\title{On two examples by Iyama and Yoshino}
\begin{document}
\begin{abstract}
 In a recent paper Iyama and Yoshino consider two interesting
 examples of isolated singularities over which it is possible to
 classify the indecomposable maximal Cohen-Macaulay modules in terms
 of linear algebra data. In this paper we present two new approaches
 to these examples. In the first approach we give a relation with
 cluster categories. In the second approach we use Orlov's result on
 the graded singularity category. 
\end{abstract}
\maketitle
\setcounter{tocdepth}{1}
\tableofcontents
\section{Introduction}
Throughout $k$ is a field.
In \cite{IY} Iyama and Yoshino consider the following two settings.
\begin{example}
\label{ref-1.1-0}
Let $S=k[[x_1,x_2,x_3]]$ and let $C_3=\langle\sigma\rangle$ be
the cyclic group of three elements. Consider the action of $C_3$
on $S$ via $\sigma x_i=\omega x_i$ where $\omega^3=1$, $\omega\neq 1$. Put
$R=S^{C_3}$.
\end{example}
\begin{example}
\label{ref-1.2-1}
Let $S=k[[x_1,x_2,x_3,x_4]]$ and let $C_2=\langle\sigma\rangle$ be
the cyclic group of two elements. Consider the action of $C_2$
on $S$ via $\sigma x_i=- x_i$.
Put $R=S^{C_2}$.
\end{example}
In both examples Iyama and Yoshino reduce the classification of
maximal Cohen-Macaulay modules over $R$ to the representation theory
of certain generalized Kronecker quivers.  They use this to classify
the rigid Cohen-Macaulay modules over~$R$.  As predicted by
deformation theory, the latter are described by discrete data.

The explicit description of the stable category of maximal
Cohen-Macaulay modules over a commutative Gorenstein ring (also known
as the singularity category \cite{Buchweitz, BuchweitzAlone,Orlov}) is
a problem that has received much attention over the years.  This appears
to be in general a difficult problem and perhaps the best one can
hope for is a reduction to linear algebra, or in other words: the
representation theory of quivers. This is precisely what
Iyama and Yoshino have accomplished.

  The proofs
of Iyama and Yoshino are based on the machinery of
mutation in triangulated categories, a general theory developed by them.  In the
current paper we present two alternative approaches to the examples.
Hopefully the additional insight obtained in this way may  be
useful elsewhere.

\medskip

Our first approach applies to Example \ref{ref-1.2-1} and is inspired
by the treatment in \cite{KeRe} of Example \ref{ref-1.1-0} where the
authors used the fact that in this case the stable category $\uMCM(R)$
of maximal Cohen-Macaulay $R$-modules is a $2$-Calabi-Yau category
which has a cluster tilting object whose endomorphism ring is the path
algebra $kQ_3$ of the Kronecker quiver with $3$ arrows. Then they
invoke their acyclicity result (slightly specialized):
\begin{theorem} \cite[\S1, Thm]{KeRe} 
Assume that $\Tscr$ is $k$-linear algebraic Krull-Schmidt
 $2$-Calabi-Yau category with a cluster tilting object $T$ such
that $A=\End(T)$ is hereditary. Then there is an exact
 equivalence betweem $\Tscr$ and the orbit category
 $D^b(\mod(A))/(\tau[-1])$.
\end{theorem}
From this result they
obtain immediately that $\uMCM(R)$ is the orbit
category $D^b(\mod(kQ_3))/(\tau[-1])$. This gives a very satisfactory
description of $\uMCM(R)$ and implies in particular the results by
Iyama and Yoshino.

\medskip

In the first part of this paper we show that Example \ref{ref-1.2-1} is amenable to a
similar approach.  Iyama and Yoshino prove that $\uMCM(R)$ is a
$3$-Calabi-Yau category with a $3$-cluster tilting object $T$ such
that $\End(T)=k$ \cite[Theorem 9.3]{IY}. We show that
under these circumstances there is an analogue of the acyclicity result
of the first author and Reiten.
\begin{theorem} \label{ref-1.3-2} (see \S\ref{ref-3.4-18}) Assume that $\Tscr$ is $k$-linear algebraic Krull-Schmidt
 $3$-Calabi-Yau category with a $3$-cluster tilting object $T$ such
 that $\End(T)=k$.  Then there is an exact
 equivalence of $\Tscr$ with the orbit category
 $D^b(\mod(kQ_n))/(\tau^{1/2}[-1])$, $n=\dim \Ext^{-1}_\Tscr(T,T)$, where $Q_n$
 is the generalized Kronecker quiver with $n$ arrows and $\tau^{1/2}$
 is a natural square root of the Auslander-Reiten translate of
 $D^b(\mod(kQ_n))$, which on the pre-projective/pre-injective component
 corresponds to ``moving one place to the left''.
\end{theorem}
In the second part of this paper, which is logically independent of
the first  we give yet another approach to the examples \ref{ref-1.1-0},\ref{ref-1.2-1} based on the following observation which might have independent
interest.  
\begin{proposition} \label{ref-1.4-3} (see Prop.\ \ref{gradedcase}) Let $A=k+A_1+A_2\cdots$
 be a finitely generated commutative graded Gorenstein $k$-algebra with an isolated singularity. Let $\widehat{A}$
 be the completion of $A$ at $A_{\ge 1}$.  Let $\uMCM_{\gr}(A)$ be the stable
 category of graded maximal Cohen-Macaulay $A$-modules. Then the
 obvious functor $\uMCM_{\gr}(A)\r \uMCM(\widehat{A})$ induces an
 equivalence
\begin{equation}
\label{ref-1.1-4}
 \uMCM_{\gr}(A)/(1) \cong \uMCM(\widehat{A})
\end{equation}
where $M\mapsto M(1)$
is the
 shift functor for the grading.
\end{proposition}
In this proposition the quotient $\uMCM_{\gr}(A)/(1)$ has to be
understood as the triangulated/Karoubian hull (as explained in
\cite{Keller6}) of the naive quotient of $\uMCM_{\gr}(A)$ by the shift
functor $?(1)$. This result is similar in spirit to \cite{A3} which
treats the finite representation type case.  Note however that one of
the main results in loc.\ cit.\ is that in case of finite
representation type case \emph{every}\/ indecomposable maximal
Cohen-Macaulay $\hat{A}$-module is gradable. This does not seem to be a formal
consequence of Proposition \ref{ref-1.4-3}. It would be interesting to investigate this further. 

In \S\ref{ref-8-57} we show that at least rigid Cohen-Macaulay modules are
always gradable so they are automatically in the image of $\MCM_{\gr}(A)$.
We expect this to be well known in some form but we have been unable to locate
a reference.

\medskip

Hence in order to understand $\MCM(\widehat{A})$ it is sufficient to
understand $\MCM_{\gr}(A)$. The latter is the graded singularity
category \cite{Orlov2} of $A$ and by \cite[Thm 2.5]{Orlov2} it
is related to $D^b(\coh(X))$ where $X=\Proj A$.

In Examples \ref{ref-1.1-0},\ref{ref-1.2-1} $R$ is the completion
of a graded ring $A$ which is the Veronese of a polynomial ring. Hence
$\Proj A$ is simply a projective space. Using
Orlov's results and the existence of exceptional collections
on projective space we get very quickly
in Example \ref{ref-1.1-0}
\[
\uMCM_{\gr}(A)\cong D^b(\mod(kQ_3))
\]
and in Example \ref{ref-1.2-1}
\[
\uMCM_{\gr}(A)\cong D^b(\mod(kQ_6))
\]
(where here and below $\cong$ actually stands for a quasi-equivalence between
the underlying DG-categories).
Finally it suffices to observe that in Example \ref{ref-1.1-0} we have
$?(-1)=\tau[-1]$ and in Example \ref{ref-1.2-1} we have $?(-1)=\tau^{1/2}[-1]$
(see \S\ref{ref-7-51} below).

\medskip

%
%

Finally we mention the following interesting side result 
\begin{proposition}  \label{ref-1.5-5}
\label{ref-1.5-6} 
Let $(R,m)$ be a Gorenstein local ``G-ring'' \textup{(}for example $R$ may be
essentially of finite type over a field\textup{)} with an isolated
singularity. 
 Then the natural functor
\begin{equation}
\label{ref-1.2-7}
\widehat{R}\otimes_R ?:\uMCM(R)\r \uMCM(\widehat{R})
\end{equation}
is an equivalence up to direct summands.  In partular every
maximal Cohen-Macaulay module over $\widehat{R}$ is a direct summand of
the completion of a maximal Cohen-Macaulay module over $R$.
\end{proposition}
The original proof (by
the first and the third author) of this result was unnecessarily complicated. After the paper was put on the arXiv
Daniel Murfet (who has become the second author) informed us about the existence of a
much nicer proof in the context of
singularity categories (see Proposition \ref{theorem:first_theorem}). The
same argument also applies to Proposition \ref{ref-1.4-3}. So we
dropped our original proofs and put the new argument in an appendix to
which we refer.

Meanwhile Orlov \cite{Orlov04} has proved (independently and using
different methods) a very general result which implies in particular
Proposition \ref{ref-1.5-6}.


\section{Acknowledgement}

We thank Osamu Iyama, Idun Reiten, Srikanth Iyengar, and Amnon Neeman for commenting on a preliminary version of this manuscript. The second author thanks Kyoji Saito, Kazushi Ueda and Osamu Iyama for pointing out the properties of the Henselization in Remark \ref{remark:henselization}, and Apostolos Beligiannis for discussing his work \cite{BeligiannisContra}.

\section{Notations and conventions}
We hope most notations are self explanatory but nevertheless we list
them here. If $R$ is a ring then $\Mod(R)$ and $\mod(R)$ denote
respectively the category of all left $R$-modules and the full
subcategory of finitely generated $R$-modules. The derived category of
all $R$-modules is denoted by $D(R)$.  If $R$ is graded then we use
$\Gr(R)$ and $\gr(R)$ for the category of graded left modules and its
subcategory of finitely generated modules. The shift functor on $\Gr(R)$
is denoted by $?(1)$. Explicitly $M(1)_i=M_{i+1}$. 
If we want to refer to
right modules then we use $R^\circ$ instead of $R$. If $X$ is a scheme
then $\Qch(X)$ is the category of quasi-coherent $\Oscr_X$-modules. If
$X$ is noetherian then $\coh(X)$ is the category of coherent
$\Oscr_X$-modules. We are generally very explicit about which
categories we use. E.g.\ we write $D^b(\mod(R))$ rather than something
like $D^b_f(R)$. If $R$ is graded and $M$, $N$ are graded $R$-modules
then $\Ext^i_R(M,N)$ is the ungraded $\Ext$ between $M$ and $N$.  If we need  $\Ext$ in the category
of graded $R$-modules then we write $\Ext^i_{\Gr(R)}(M,N)$.
\section{First approach to the second example}
\subsection{Some preliminaries on tilting complexes}
Let $C,E$ be rings. We denote the unbounded derived category of right
$C$-modules by $D(C^\circ)$.
\def\Eq{\operatorname{Eq}}
\def\Tilt{\operatorname{Tilt}}%
We let $\Eq(D(C^\circ),D(E^\circ))$ be the set of triangle equivalences
of $D(C^\circ)\r D(E^\circ)$ modulo natural isomorphisms.
Define
$\Tilt(C,E)$ as the set of pairs $(\phi,T)$ where $T$ is a perfect
complex generating $D(E^\circ)$ and $\phi$ is an isomorphism $C\r
\RHom_E(T)$.  Associated to $(\phi,T)\in \Tilt(C,E)$ there is a
canonical equivalence $\theta:D(C^\circ)\r D(E^\circ)$ such that
$\theta(C)=T$.  It may be constructed either directly
\cite{Ri} or using DG-algebras \cite{Keller10}.  The induced map
\[
\Tilt(C,E)\r \Eq(D(C^\circ),D(E^\circ))
\]
is obviously injective (as it is canonically split), but not known to be
surjective. Below we will  informally refer to the elements of
$\Tilt(C,E)$ as tilting complexes.
\subsection{A square root of $\tau$  for a
 generalized Kronecker quiver}
\label{ref-3.2-8}
Let $W$ be a finite dimensional $k$-vector space and
let $C$ be the path algebra of the quiver\footnote{We use the convention that multiplication in the path algebra is concatenation. So representations
correspond to right modules. }
\begin{equation}
\label{ref-3.1-9}
\psfrag{V}[][]{$W$}
\psfrag{1}[][]{$1$}
\psfrag{2}[][]{$2$}
\includegraphics[width=3cm]{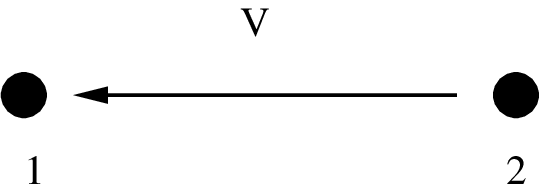}
\end{equation}
Let $E$ be the path algebra of the quiver
\[
\psfrag{V}[][]{$W^\ast$}
\psfrag{1}[][]{$1$}
\psfrag{2}[][]{$2$}
\includegraphics[width=3cm]{beilinson_mini}
\]
which we think of as being obtained from \eqref{ref-3.1-9} by ``inverting the arrows'' and renumbering the vertices $(1,2)\mapsto (2,1)$. 

Let $P_i$,$I_i,S_i$ be respectively the projective, injective and simple right
$C$-module corresponding to vertex $i$. For $E$ we use $P'_i$, $I'_i$, $S'_i$.
Let $r_i:\mod(C^\circ)\r\mod(E^\circ)$ be the reflection functor at vertex $i$. Recall that if
$(U,V)$ is a representation of $C$ then $r_1(U,V)$ is given by $(V,U')$
where $U'=\ker(V\otimes W\r U)$ (taking into account the renumbered vertices).

The right derived functor $Rr_1$ of $r_1$ defines an
equivalence $D(C^\circ) \r D(E^\circ)$. It is obtained from the tilting
complex $S'_2[-1]\oplus P'_1$ \cite{ARP}. One has  (see \cite{Gabriel2})
\begin{equation}
\label{ref-3.2-10}
Rr_1\circ Rr_1=\tau_C
\end{equation}
where $\tau_C$ is the Auslander Reiten translate on $D(C^\circ)$. 
Assume now that $W$ is equipped with an isomorphism $\pi:W\r
W^\ast$.  Then $\pi$ yields an equivalence
$D(E^\circ)\cong D(C^\circ)$, which we denote by the same symbol. We use
the same convention for the transpose isomorphism $\pi^\ast:W\r W^\ast$.
\begin{lemmas}
\label{ref-3.2.1-11}
We have $r_1\circ \pi^{-1}=\pi^\ast\circ r_1$ as functors
$D(C^\circ)\r D(C^\circ)$.
\end{lemmas}
\begin{proof}
 Let $(U,V)$ be a representation of $C$ determined by a linear map
 $c:V\otimes W\r U$ and put $(V,U'')=(r_1\circ \pi^{-1})
 (U,V)$. Then one checks that $U''$ is given by the exact sequence
\[
0\r U''\r V\otimes W^\ast \xrightarrow{c\circ (\pi^{-1}\otimes \id)} U
\r 0
\]
where the first non-trivial map induces the action $U''\otimes W\r
V$. Similarly if we put $(V,U')=(\pi^\ast\circ r_1)
 (U,V)$ then one gets the same sequence
\[
0\r U'\r V\otimes W^\ast\xrightarrow{c\circ (\pi^{-1}\otimes \id)} U
\]
where the first non-trivial map again yields the action $U'\otimes W \r
V$. Thus we have $(V,U')=(V,U'')$.
\end{proof}
Below we
put $a=\pi\circ Rr_1$.
\begin{lemmas} \label{ref-3.2.2-12} One has $(\pi^\ast\circ \pi^{-1})\circ a^2=\tau$.
In particular $\tau\cong a^2$ if and only
if $\pi$ is self-adjoint or anti self-adjoint.
\end{lemmas}
\begin{proof} This is a straightforward verification using
Lemma \ref{ref-3.2.1-11} and \eqref{ref-3.2-10}.
\end{proof}
For use below we record
\begin{align*}
aP_2&=P_1\\
aP_1&=I_2[-1]\\
aI_2&=I_1
\end{align*}
\subsection{A $3$-Calabi-Yau category with a $3$-cluster tilting object}
\label{ref-3.3-13}
We let the notations be as in the previous section,

Put $\Hscr=D^b(\mod(C^\circ))$, $\Dscr=\Hscr/a[-1]$.
As $\Hscr$ is hereditary we have
\[
\Ind(\Dscr)=\Ind(\Hscr)/a[-1]
\]
Inspection reveals that
\begin{equation}
\label{ref-3.3-14}
\Ind(\Dscr)=\Ind(\Hscr)\cup \{ I_2[-1]\}
\end{equation}
\begin{lemmas} \label{ref-3.3.1-15} $\Dscr$ is $3$-Calabi-Yau if and only if $\pi$
is self-adjoint or anti self-adjoint.
\end{lemmas}
\begin{proof} Let $S$ be the Serre-functor for $\Hscr$. Being canonical
$S$ commutes with the auto-equivalence $a[-1]$. Hence $S$  induces
an autoequivalence on $\Dscr$ which is easily seen to be the Serre
functor of $\Dscr$.

In $\Dscr$ we have
 $S=\tau[1]=(\pi^\ast\circ \pi^{-1})\circ
 a^2[1]=(\pi^\ast\circ \pi^{-1})[3]$. Thus $\Dscr$ is
 $3$-Calabi-Yau if and only if $\pi^\ast\circ \pi^{-1}$ is
 isomorphic to the identity functor. It is easy to see that this is the
case if and only is $\pi^\ast\circ \pi^{-1}=\pm 1$ in $\End_k(W)$.
\end{proof}
\begin{lemmas}
\label{ref-3.3.2-16}
The object $P_1$ in $\Dscr$ satisfies
\[
\Ext^i_\Dscr(P_1,P_1)=0 \mbox{ for } i=1,2, \quad \Hom_\Dscr(P_1,P_1)=k
\mbox{ and } \Ext^{-1}_{\Dscr}(P_1,P_1)=W.
\]
\end{lemmas}
\begin{proof}
For $N\in \Ind(\Hscr)\cup \{I_2[-1]\}$ one  computes
\begin{equation}
\label{ref-3.4-17}
\Hom_\Dscr(P_1,N)=\Hom_\Hscr(P_1,N)
\end{equation}
Thus we find
\begin{align*}
\Hom_\Dscr(P_1,P_1[-1])&=\Hom_\Dscr(P_1,a^{-1}P_1)\\
&=\Hom_\Dscr(P_1,P_2)\\
&=W
\end{align*}
\[
\Hom_\Dscr(P_1,P_1)=k
\]
\begin{align*}
\Hom_\Dscr(P_1,P_1[1])&=\Hom_\Dscr(P_1,a P_1)\\
&=\Hom_\Dscr(P_1,I_2[-1])\\
&=0
\end{align*}
and
\begin{align*}
\Hom_\Dscr(P_1,P_1[2])&=\Hom_\Dscr(P_1,aP_1[1])\\
&=\Hom_\Dscr(P_1,I_2)\\
&=0\qed
\end{align*}
\def\qed{}\end{proof}
The following lemma is not used explicitly.
\begin{lemmas} The object $P_1$ in $\Dscr$ has the properties of a
 $3$-cluster tilting object, i.e.\ if $\Ext^i_\Dscr(P_1,N)=0$ for $i=1,2$ then $N$ is a
 sum of copies of $P_1$.
\end{lemmas}
\begin{proof}
Assume that $N\in \Ind(\Hscr)\cup \{I_2[-1]\}$ is such that
$\Hom_\Dscr(P_1,N[1])=\Hom_\Dscr(P_1,N[2])=0$. We have to prove $N=P_1$.

We may rewrite
\begin{align*}
\Hom_\Dscr(P_1,N[2])&=\Hom_\Dscr(P_1[-1],N[1])\\
&=\Hom_\Dscr(a^{-1}P_1,N[1])\\
&=\Hom_\Dscr(P_2,N[1])
\end{align*}
Thus we find $\Hom_\Dscr(P_1,aN)=\Hom_\Dscr(P_2,aN)=0$.
Hence $aN\not\in \Ind(\Hscr)$. We deduce $N\in \{P_1,I_2[-1]\}$.

But if $N=I_2[-1]$ then
\begin{align*}
\Hom_\Dscr(P_1,N[2])&=\Hom_\Dscr(P_1,I_2[1])\\
&=\Hom_\Dscr(P_1,aI_2)\\
&=\Hom_\Dscr(P_1,I_1)\\
&\neq 0
\end{align*}
So we are left with the possibility $N=P_1$ which finishes the proof.
\end{proof}
\subsection{Proof of  Theorem \ref{ref-1.3-2}}
\label{ref-3.4-18}
Let $\Tscr$ be an algebraic $\Ext$-finite Krull-Schmidt
$3$-Calabi-Yau category containing a $3$-cluster tilting object $T$ such
that $\End_\Tscr(T)=k$.
\begin{lemmas}
\label{ref-3.4.1-19}
 Let $N\in \Tscr$. Then there exists a distinguished triangle in $\Tscr$
\begin{equation}
\label{ref-3.5-20}
T^a\r T^b\oplus T[-1]^c\r N[1]\r
\end{equation}
\end{lemmas}
\begin{proof}
Let $Y$ be defined (up to isomorphism) by the following distinguished triangle\footnote{It would be more logical to write e.g. $\Ext^1_\Tscr(T,N)\otimes_k T$
for $T^{\Ext^1_\Tscr(T,N)}$ but this would take a lot more space.}
\[
Y\r T^{\Ext^1_\Tscr(T,N)}\oplus T[-1]^{\Ext^2_\Tscr(T,N)} \r N[1] \r
\]
A quick check reveals that $\Ext^1_\Tscr(T,Y)=\Ext^2_\Tscr(T,Y)=0$. Hence
$Y=T^a$ for some~$a$.
\end{proof}
We need to consider the special case $N=T[1]$. Then the distinguished triangle \eqref{ref-3.5-20}
(constructed as in the proof) has the form
\begin{equation}
\label{ref-3.6-21}
T^{\Ext^{-1}_\Tscr(T,T)}\xrightarrow{\phi} T[-1]\xrightarrow{\alpha} T[2]
\xrightarrow{\beta}
\end{equation}
where $\phi$ is the universal
map (this follows from applying $\Hom_\Tscr(T,-)$).
Since $\End_\Cscr(T[2])=k$ it follows that $\alpha$,$\beta$ are determined
up to (the same) scalar.

This has a surprizing consequence. Applying $\Hom_\Tscr(-,T)$ to the triangle
\eqref{ref-3.6-21} we find that $\Hom_\Tscr(\beta[-1],T)^{-1}$ defines an
isomorphism
\[
\pi:\Ext^{-1}_\Tscr(T,T)\r \Ext^{-1}_\Tscr(T,T)^\ast
\]
Thus $W\overset{\text{def}}{=}\Ext^{-1}_\Cscr(T,T)$ comes equipped
with an isomorphism $\pi:W\r W^\ast$ which is canonical up to a
scalar.  In other words we are in the setting of \S\ref{ref-3.2-8}
and we now use the notations introduced in sections \ref{ref-3.2-8}
and \ref{ref-3.3-13}.

As $a$ is obtained from the reflection in vertex $1$, one verifies
(see \S\ref{ref-3.2-8}) that $a$ is associated to the element of
$\Tilt(C,C)$ given by $(\theta,I_2[-1]\oplus P_1)$ where $\theta:C\r
\End_C(I_2[-1]\oplus P_1)$ is the composition
\begin{equation}
\label{ref-3.7-22}
C=\begin{pmatrix}
k&0\\
W&k
\end{pmatrix}
\xrightarrow{\pi}
\begin{pmatrix}
k&0\\
W^\ast&k
\end{pmatrix}
=\End_C(I_2[-1]\oplus P_1)
\end{equation}

Since the autoequivalence $a$ is a derived functor that commutes with
coproducts it is isomorphic to a derived tensor functor  $-\Lotimes_{C} X$
for some $X\in D(C^e)$, by \cite[6.4]{Keller1}.  As a
right $C$-module we have $X\cong I_2[-1]\oplus P_1$.

Now we use the assumption that $\Hscr$ is algebraic and we proceed
more or less as in the appendix to \cite{KeRe}.  By
\cite[Thm. 4.3]{Keller1} we may assume that $\Tscr$ is a strict (= closed under
isomorphism) triangulated subcategory of a derived category $D(\Ascr)$
for some DG-category $\Ascr$. We denote by ${}_C\Tscr$ the full
subcategory of $D(C\otimes\Ascr)$ whose objects are differential
graded $C\otimes\Ascr$-modules which are in $\Tscr$ when considered as
$\Ascr$-modules. Clearly ${}_C \Tscr$ is triangulated. By
\cite[Lemma A.2.1(a)]{KeRe} $T$ may be lifted to an object
in ${}_C\Tscr$, which we also denote by $T$. Put $S=T\oplus T[-1]$.
\begin{lemmas}
\label{ref-3.4.2-23}
One has an isomorphism in ${}_C\Tscr$
\[
X\Lotimes_B S\cong S[1]
\]
\end{lemmas}
\begin{proof}
As objects in $\Tscr$ we have
\begin{align*}
X\Lotimes_C S&=(I_2[-1]\oplus P_1)\Lotimes_C S\\
&=I_2\Lotimes_C S[-1]\oplus P_1\Lotimes_C S
\end{align*}
Clearly  $P_1\Lotimes_C S\cong T$. To compute $I_2\Lotimes_C S$ we use
the resolution
\[
0\r P_1^{\Ext_\Tscr^{-1}(T,T)}\r P_2\r I_2\r 0
\]
Tensoring with $S$ we get a distinguished triangle
\[
T^{\Ext_\Tscr^{-1}(T,T)}\r T[-1]\r I_2\Lotimes_C S\r
\]
Comparing with \eqref{ref-3.6-21} we find $I_2\Lotimes_C S\cong T[2]$.
Thus, we have indeed an isomorphism
\[
\varphi: X\Lotimes_B S\r S[1]
\]
in $\Tscr$.

Now we check that $\varphi$ is $C$-equivariant in $\Tscr$.
The left $C$-module structure on $X\Lotimes_B S$ is obtained from
the (homotopy) $C$-action on $I_2[-1]\oplus P_1$ as given in
\eqref{ref-3.7-22}.

Let $\mu$
be an element of $W=\Hom_C(P_1,P_2)=\Ext^{-1}_\Tscr(T,T)$. We need to
prove that the following diagram is commutative in $\Tscr$.
\[
\begin{CD}
I_2[-1]\Lotimes_B S @>\cong >> T[1]\\
@V \pi(\mu)\Lotimes_B \Id_S VV @VV\mu V\\
P_1 \Lotimes_B S @>>\cong > T
\end{CD}
\]
We write this out in triangles
\[
\begin{CD}
T^{\Ext^{-1}(T,T)} @>\phi>> T[-1] @>\alpha>> T[2]@>\beta>> \\
@V \pi(\mu) VV @VVV @VV\mu V\\
T @>>> 0 @>>> T[1]@>>\id>
\end{CD}
\]
Rotating the triangles we need to prove that the following square is commutative
\[
\begin{CD}
T[1] @>\beta[-1]>> T^{\Ext^{-1}(T,T)}\\
@V\mu VV @VV\pi(\mu)V\\
T @= T
\end{CD}
\]
This commutivity holds precisely because of the definition of $\pi$.
So $\phi$ is indeed $C$-equivariant.

But according to \cite[Lemma A.2.2]{KeRe},
any $C$-equivariant morphism in
$\Tscr$ between objects in ${}_C\Tscr$ may be lifted to a morphism in
${}_C\Tscr$. This finishes the proof.
\end{proof}
We now have a functor
\[
?\Lotimes_C T:\Cscr\r \Tscr
\]
and by Lemma \ref{ref-3.4.2-23} one finds that  $a[-1](?)\Lotimes_C T$ is isomorphic to
$?\Lotimes_C T$. By
the universal property of orbit categories \cite{Keller6} we
obtain a triangulated functor
\[
Q:\Dscr\r \Tscr
\]
which sends $P_1$ to $T$.
\begin{lemmas} $Q$ is an equivalence.
\end{lemmas}
\begin{proof}
We observe that analogues of the distinguished triangles
\eqref{ref-3.5-20} exist in $\Dscr$ (with $P_1$ replacing $T$). Indeed, let
$N\in \Ind(\Dscr)$. By \eqref{ref-3.3-14} we have $N\in
\Ind(\Hscr)\cup \{I_2[-1]\}$. If $N\in \Ind(\Hscr)$ then $N[1]\cong
aN$ and the analog of \eqref{ref-3.5-20} is simply the image in $\Dscr$ of the
projective resolution of $aN$ in $\Hscr$ (taking into account that
$P_2=a^{-1}P_1=P_1[-1]$).

If $N=I_2[-1]$ then $N[1]=I_2$ and the analog of \eqref{ref-3.5-20}
is the image in $\Dscr$ of the projective resolution of $I_2$ in $\Hscr$.

To prove that $Q$ is fully faithful we have to prove that $Q$ induces
an isomorphism $\Hom_\Dscr(M,N)\r \Hom_\Tscr(QM,QN)$. Using the
analogues of \eqref{ref-3.5-20} we reduce to $M=P_1[i]$. But since
$\Hom_\Dscr(P_1[i],N)=\Hom_\Dscr(P_1[-1],N[-i-1])$ we reduce in fact to $M=P_1[-1]$.
It now suffices to apply $\Hom_\Dscr(P_1[-1],-)$ to
\[
P_1^a\r P_1^b\oplus P_1[-1]^c\r N[1]\r
\]
taking into account that $\Hom_\Dscr(M,N)\r \Hom_\Tscr(QM,QN)$ is an
isomorphism for $M=P_1$, $N=P_1$, $P_1[1]$, $P_2[2]$ by Lemma
\ref{ref-3.3.2-16}.

As a last step we need to prove that $Q$ is essentially surjective.
But this follows from the distinguished triangles \eqref{ref-3.5-20}
together with the fact that $QP_1=T$.
\end{proof}
To finish the proof of Theorem \ref{ref-1.3-2} we observe that since $\Tscr$ is
$3$-Calabi-Yau, so is $\Dscr$. Hence by Lemma \ref{ref-3.3.1-15} $\pi$
is either self-adjoint or anti self-adjoint. By Lemma \ref{ref-3.2.2-12}
we deduce $a^2\cong\tau$ and hence we may write $a=\tau^{1/2}$.
\begin{remarks} It would be interesting to deduce the fact
that $\pi$ is (anti) self-adjoint directly from the Calabi-Yau property
of $\Tscr$, without going through the construction of $\Dscr$ first. This
would have made our arguments above more elegant.
\end{remarks}
\begin{remarks} Iyama and Yoshino also consider $2n+1$-Calabi-Yau categories $\Tscr$
 equipped with a $2n+1$-cluster tilting object $T$ such that
 $\End(T)=k$ and $\Ext^{-i}(T,T)=0$ for $0<i<n$.  They relate such
 $\Tscr$ to the representation theory of the generalized Kronecker
 quiver $Q_m$ where $m=\dim \Ext^{-n}(T,T)$.

One may show that our techniques
 are applicable to this case as well and yield 
$\Tscr\cong D^b(\mod(kQ_m))/(\tau^{1/2}[-n])$. We thank Osamu Iyama
for bringing this point to our attention. 
\end{remarks}
\section{The singularity category of graded Gorenstein rings}
\subsection{Orlov's results}
\label{ref-6.1-42}
Let $A=k+A_1+A_2+\cdots$ be a commutative finitely generated graded
$k$-algebra.  As in \cite{AZ} we write $\qgr(A)$ for the quotient of
$\gr(A)$ by the Serre subcategory of graded finite length modules.  We
write $\pi:\gr(A)\r \qgr(A)$ for the quotient functor. If $A$ is
generated in degree one and $X=\Proj A$ then by Serre's theorem
\cite{Se} we have $\coh(X)=\qgr(A)$.

Now assume that $A$ is Gorenstein. Then
we have $\RHom_A(k,A)\cong k(a)[-d]$ where $d$ is the Krull dimension of $R$
and $a\in \ZZ$. The
number $a$ is called the Gorenstein parameter of $A$ (see
\cite[Definition 2.1]{Orlov2}).
\begin{examples} If $A$ is a polynomial ring in $n$ variables (of degree
 one) then $d=n$, $a=n$.
\end{examples}
For use below we record another incarnation
of the Gorenstein parameter. Let $A'$ be the graded $k$-dual of $A$. Then
\begin{equation}
\label{ref-6.1-43}
R\Gamma_{A_{{>0}}}(A)\cong A'(a)[-d]
\end{equation}
where $R\Gamma_{A_{>0}}$ denotes cohomology with support in the ideal $A_{>0}$.

The following is a particular
case of \cite[Thm 2.5]{Orlov2}.
\begin{theorems} \label{ref-6.1.2-44}
If $a\ge 0$ then there are fully faithful functors
\[
\Phi_i:\uMCM_{\gr} (A)\r D^b(\qgr(A))
\]
such that for $\Tscr_i= \Phi_i \uMCM_{\gr}(A)$ there is a
semi-orthogonal decomposition
\[
D^b(\qgr(A))=\langle \pi A(-i-a+1),\ldots, \pi A(-i), \Tscr_i\rangle
\]
\end{theorems}
Hence under the hypotheses of the theorem we obtain in particular that
\[
\uMCM_{\gr}(A)\cong {}^\perp\langle \pi A(-i-a+1),\ldots, \pi A(-i)\rangle
\subset D^b(\qgr(A))
\]
for arbitrary $i$.
\subsection{The action of the shift functor on the singularity category}
\label{ref-6.2-45}
Unfortunately the functors $\Phi_i$ introduced in the previous section
are not compatible with $?(1)$. Our aim in this section is to understand
how $?(1)$ acts on the image of $\Phi_i$. This requires us to dig
deeper into Orlov's construction which has the unusual
feature of depending on the category $D^b(\gr_{\ge i}A)$ where
$\gr_{\ge i}A$ are the finitely generated graded $A$-modules with non zero
components
concentrated in degrees $\ge i$. The quotient functor
\[
D^b(\gr_{\ge i}A)\hookrightarrow D^b(\gr A)\xrightarrow{\pi} D^b(\qgr A)
\]
has a right adjoint $R\omega_i A$. Its image is denoted by $\Dscr_i$.

We let $P_i$
be the graded projective $A$-module of rank one generated in degree
$i$ (i.e.\ $P_i=A(-i)$). Likewise $S_i$ is the simple $A$-module concentrated in degree
$i$. As in \cite{Orlov2} we put $\Pscr_{\ge i}=\langle (P_j)_{j\ge i}\rangle$,
$\Sscr_{\ge i}=\langle (S_j)_{j\ge i}\rangle$ and obvious variants
with other types of inequality signs. In \cite{Orlov2} it is proved
that the image $\Tscr_i$ of $\Phi_i$ is the left orthogonal to $\Pscr_{\geq i}$ inside
$D^b(\gr_{\ge i}A)$. The identification of $\Tscr_i$ with the graded singularity
category is through the composition
\begin{equation}
\label{ref-6.2-46}
\Tscr_i\cong D^b(\gr_{\ge i}A)/\Pscr_{\geq i}\cong D^b(\gr A)/\operatorname{perf}(A)
\cong \uMCM_{\gr }(A)
\end{equation}
Assume $a\ge 0$. Then the relation between $\Tscr_i$, $\Dscr_i$ is given
by the following semi-orthogonal decompositions
\[
D^b(\gr A)=\langle \Sscr_{<i}, \overbrace{\Pscr_{\ge i+a},
\underbrace{P_{i+a-1},\ldots,P_i,\Tscr_i}_{\Dscr_i\cong D^b(\qgr(A))}}^{D^b(\gr_{\ge i}A)}\rangle
\]
This is a refinement of Theorem \ref{ref-6.1.2-44}.

The category $\uMCM_{\gr}(A)$ comes equipped with the shift functor
$?(1)$.  We denote the induced endofunctor on $\Tscr_i$ by $\sigma_i$. We will
now compute it. \def\cone{\operatorname{cone}}
\begin{lemmas} For $M\in \Tscr_i\subset D^b(\qgr(A))$ we have
 \begin{equation}
\label{ref-6.3-47}
\sigma_i M=\cone(\RHom_{\qgr(A)}(\pi A(-i),M)\otimes_k
\pi A(-i+1)\r  M(1))
\end{equation}
where the symbol ``cone'' is to be understood in a functorial sense,
for example by computing it on the level of complexes  after first replacing $M$ by an injective
resolution.
\end{lemmas}
\begin{proof}
Let $N\in \Tscr_i\subset D^b(\gr(A))$. To compute $\sigma_i N$ we see by \eqref{ref-6.2-46}
that we have
to find $\sigma_i N\in \Tscr_i$ such that $\sigma_i N\cong N(1)$ up to projectives. It
is clear we should take
\[
\begin{aligned}
\sigma_i N&=\cone (\RHom_{\gr(A)}(P_{i-1},N(1))\otimes_k P_{i-1}\r N(1))\\
&=\cone (\RHom_{\gr(A)}(P_{i},N)\otimes_k P_{i-1}\r N(1))
\end{aligned}
\]
Now we note that the $\RHom$ can be computed in $\Dscr_i\cong
D^b(\qgr(A))$. Furthermore since the result lies in
$\Tscr_i\subset\Dscr_i$ we can characterize it uniquely by applying
$\pi$ to it. Since $\pi$ commutes with $?(1)$ we obtain
\eqref{ref-6.3-47} with $M=\pi N$.
\end{proof}

\subsection{The Serre functor for a graded Gorenstein ring}
\label{ref-6.3-48}
Let $A,a,d$ be as above but now assume that $A$ has an isolated
singularity and let $M,N\in \uMCM_{\gr}(A)$. Then by a variant of \cite[Thm
8.3]{IY} we have a canonical graded isomorphism
\[
\Ext^d_A(\underline{\Hom}_A(M,N),A)\cong \underline{\Hom}_A(N,M[d-1])
\]
and furthermore an appropriate version of local duality yields
\[
\Ext^d_A(\underline{\Hom}_A(M,N),A)=\underline{\Hom}_A(M,N)^\ast(a)
\]
In other words we find
\[
\underline{\Hom}_A(M,N)^\ast=\underline{\Hom}_A(N,M[d-1](-a))
\]
and hence the Serre functor $S$ on $\uMCM(A)$ is given by $?[d-1](-a)$.

It is customary to write $S=\tau[1]$ so that we have the usual formula
\[
\underline{\Hom}_A(M,N)^\ast=\Ext^1(N,\tau M)
\]
In this setting we find
\begin{equation}
\label{ref-6.4-49}
\tau=?[d-2](-a)
\end{equation}
\subsection{The Gorenstein parameter of a Veronese subring}
We remind the reader of the following well-known result.
\begin{propositions}
\label{ref-6.4.1-50}
Let $B$ be a polynomial ring in $n$ variables of degree
one. Assume $m\mid n$ and let $B^{(m)}$ be the corresponding Veronese
subring of $B$. I.e.\ $B^{(m)}_i=B_{mi}$. Then $B^{(m)}$ is Gorenstein with
Gorenstein parameter $n/m$.
\end{propositions}
\begin{proof} The Gorenstein property is standard. To compute
the Gorenstein invariant we first let $A$ be the ``blown up'' Veronese.
I.e.
\[
A_i=
\begin{cases}
B_i&\text{if $m\mid i$}\\
0&\text{otherwise}
\end{cases}
\] Let $a,b=n$ be respectively the Gorenstein parameters of
$A$ and $B$. If $M$ is a $B$-module write $M^+$ for $\oplus_i M_{mi}$,
considered as graded $A$-module. We have
\begin{align*}
A'(a)[-n]&=R\Gamma_{A_{>0}}(A)\qquad (\text{see \eqref{ref-6.1-43}})\\
&=R\Gamma_{A_{>0}}(B)^+\\
&=R\Gamma_{B_{>0}}(B)^+\\
&=(B'(b)[-n])^+\\
&=A'(b)[-n]
\end{align*}
In the 3rd equality we have used that local homology is insensitive to
finite extensions. We deduce $a=b=n$. Since $B^{(m)}$ is obtained from
$A$ by dividing the grading by $m$ obtain $n/m$ as Gorenstein parameter
for $B^{(m)}$.
\end{proof}
\begin{remarks}
 In characteristic zero we could have formulated the result for
 invariant rings of finite subgroups of $\Sl_n(k)$ (with the same
 proof).  However in finite characteristic Veronese subrings are not
 always invariant rings (consider the case where the characteristic
 divides $m$).
\end{remarks}
\section{The Iyama-Yoshino examples (again)}
\label{ref-7-51}
\subsection{Example \ref{ref-1.1-0}}
\label{ref-7.1-52}
Let $B=k[x_1,x_2,x_3]$ and $A=B^{(3)}$. We have
$X\overset{\text{def}}{=} \Proj A=\Proj B =\PP^2$.  By Proposition \ref{ref-6.4.1-50}
$A$ has Gorenstein invariant $1$.

Unfortunately we have to deal with the unpleasant notational
problem that the shift functors on $\coh(\PP^2)$ coming
from $A$ and $B$ do not coincide. To be consistent with the
sections \ref{ref-6.1-42},\ref{ref-6.2-45} we will denote them respectively by $?(1)$ and $?\{1\}$. Thus
$?(1)=?\{3\}$. Note that this choice is rather unconventional.

According to Theorem \ref{ref-6.1.2-44}
we have a semi-orthogonal decomposition
\[
D^b(\coh(X))=\langle \Oscr_{\PP^2},\Tscr_0\rangle
\]
From the fact that $D^b(\coh(X))$ has a strong exceptional collection
$\Oscr_{\PP^2}$, $\Oscr_{\PP^2}\{1\}$. $\Oscr_{\PP^2}\{2\}$ we deduce that
there is a semi-orthogonal decomposition
\[
\Tscr_0=\langle \Oscr_{\PP^2}\{1\},\Oscr_{\PP^2}\{2\}\rangle
\]
In particular $\RHom_{\PP^2}(\Oscr_{\PP^2}\{1\}\oplus \Oscr_{\PP^2}\{2\},-)$
defines an equivalence between $\Tscr_0$ and the representations
of the quiver $Q_3$
\[
\psfrag{V}[][]{$V$}
\psfrag{1}[][]{$1$}
\psfrag{2}[][]{$2$}
\includegraphics[width=3cm]{beilinson_mini}
\]
where $V=kx_1+kx_2+kx_3$ and where $\Oscr_{\PP^2}\{i\}$ corresponds to
the vertex labeled by $i$. By \eqref{ref-6.4-49} the Auslander-Reiten
translate on $\uMCM_{\gr}(A)$ is given by $?[1](-1)$. In other words: the shift functor
on $\uMCM_{\gr}(A)$ is given by $(\tau[-1])^{-1}$. By Proposition \ref{gradedcase}
we find (using $R=\widehat{A}$)
\[
\uMCM(R)\cong \uMCM_{\gr}(A)/(1)\cong D^b(\mod(kQ_3))/(\tau[-1])
\]
which is what we wanted to show.
\begin{remarks} Note that this in this example we had
no need for the somewhat subtle formula \eqref{ref-6.3-47}.
\end{remarks}
\subsection{Example \ref{ref-1.2-1}}
We use similar conventions as in the previous section,
Let $B=k[x_1,x_2,x_3,x_4]$ and $A=B^{(2)}$.
We have $X=\Proj A\cong \Proj B= \PP^3$ and we denote the corresponding
shift functors by $?(1)$, $?\{1\}$ so that $?(1)=?\{2\}$.  By Proposition \ref{ref-6.4.1-50}
$A$ has Gorenstein invariant $2$.
By Theorem \ref{ref-6.1.2-44} we have
a semi-orthogonal decomposition
\[
D^b(\coh(X))=\langle \Oscr_{\PP^3},\Oscr_{\PP^3}\{2\},\Tscr_{-1}\rangle
\]
Now $D^b(\coh(X))$ has a strong exceptional collection
$\Oscr_{\PP^3}$, $\Oscr_{\PP^3}\{1\}$. $\Oscr_{\PP^3}\{2\}$,
$\Oscr_{\PP^3}\{3\}$. This sequence is geometric \cite[Prop.\
3.3]{Bondal} and hence by every mutation is strongly exceptional
\cite[Thm.\ 2.3]{Bondal}. We get in particular the following strongly
exceptional collection

$\Oscr_{\PP^3}$,
$\Oscr_{\PP^3}\{2\}$. $\Omega^\ast_{\PP^3}\{1\}$,  $\Oscr_{\PP^3}\{3\}$ where $\Omega_{\PP^3}$
is defined by the exact sequence
\begin{equation}
\label{ref-7.1-53}
0\r \Omega_{\PP^3}\r V\otimes \Oscr_{\PP^3}\{-1\}\r \Oscr_{\PP^3}\r 0
\end{equation}
where $V=kx_1+kx_2+kx_3+kx_4$.  Thus there is a semi-orthogonal
decomposition
\[
\Tscr_{-1}=\langle \Omega^\ast_{\PP^3}\{1\},\Oscr_{\PP^3}\{3\}\rangle
\]
An easy computation yields
\[
\RHom_{\PP^3}(\Omega^\ast_{\PP^3}\{1\},\Oscr_{\PP^3}\{3\})=\wedge^2 V
\]

$\RHom_{\PP^3}(\Omega^\ast_{\PP^3}\{1\}\oplus \Oscr_{\PP^3}\{3\},-)$
defines an equivalence between $\Tscr_{-1}$ and the representations
of the quiver $Q_6$
\[
\label{ref-7.2-54}
\psfrag{V}[][]{$\wedge^2 V$}
\psfrag{1}[][]{$1$}
\psfrag{2}[][]{$2$}
\includegraphics[width=3cm]{beilinson_mini}
\]
Put $W=\wedge^2 V$ and choose an arbitrary trivialization $\wedge^4 V\cong k$.
Let $\pi:W\r W^\ast$ be the resulting (self-adjoint) isomorphism. We
are in the setting of \S\ref{ref-3.2-8} and hence can define $\tau^{1/2}$
as acting on the derived category of $Q_6$.

We will now compute
$\sigma_{-1}(\Omega^\ast_{\PP^3}\{1\})$, $\sigma_{-1}(\Oscr_{\PP^3}\{3\})$.
An easy computation yields
\begin{align*}
\RHom_{\PP^3}(\Oscr_{\PP^3}\{2\},\Omega^\ast_{\PP^3}\{1\})&=V^\ast\\
\RHom_{\PP^3}(\Oscr_{\PP^3}\{2\},\Oscr_{\PP^3}\{3\})&=V
\end{align*}
Using the formula \eqref{ref-6.3-47} we find
\begin{equation}
\label{ref-7.2-55}
\sigma_{-1}(\Oscr_{\PP^3}\{3\})=
\cone (V\otimes \Oscr_{\PP^3}\{4\}\r \Oscr_{\PP^3}\{5\})=\Omega_{\PP^3}\{5\}[1]
\end{equation}
\begin{equation}
\label{ref-7.3-56}
\sigma_{-1}(\Omega^\ast_{\PP^3}\{1\})=
\cone (V^\ast\otimes \Oscr_{\PP^3}\{4\}\r \Omega^\ast_{\PP^3}\{3\})=
\Oscr_{\PP^3}\{3\}[1]
\end{equation}
where in the second line we have used the dual version of \eqref{ref-7.1-53}.

Let $P_i$ be the projective representation of $Q_6$ generated in
vertex $i$. The endo\-functor on $D^b(\mod(kQ_6))$ induced by $\sigma_{-1}$ will
be denoted by the same letter. We will now compute it. From \eqref{ref-7.3-56}
we deduce immediately $\sigma_{-1}(P_1)=P_2[1]$. To analyze \eqref{ref-7.2-55}
we note
that a suitably shifted slice of the Koszul sequence
has the form
\[
0\r \wedge^4 V\otimes \Omega^\ast_{\PP^3}\{1\}\r
\wedge^2 V\otimes\Oscr_{\PP^3}\{3\}
\r \Omega_{\PP^3}\{5\}\r 0
\]
Thus $\Omega_{\PP^3}\{5\}$ corresponds to the cone of
\[
\wedge^4 V\otimes P_1\r \wedge^2 V\otimes P_2
\]
which is easily seen to be equal to $\wedge^4 V\otimes \tau^{-1}P_1$.

If we use our chosen trivialization $\wedge^4 V\cong k$ then we
see that at least on objects $\sigma_{-1}$ coincides with $\tau^{-1/2}[1]$.
It is routine to extend this to an isomorphism of functors by starting
with a bounded complex of projectives in $\mod(kQ_6)$.

By Proposition \ref{gradedcase}
we find (using $R=\widehat{A}$)
\[
\uMCM(R)\cong \uMCM_{\gr}(A)/(1)\cong D^b(\mod(kQ_6))/(\tau^{1/2}[-1])
\]
which is what we wanted to show.
\section{A remark on gradability of rigid modules}
\label{ref-8-57}
We keep notations as in the previous section.
Since in the Iyama-Yoshino examples $\uMCM_{\gr}(A)$ is the derived
category of a hereditary category the functor
\[
\uMCM_{\gr}(A)\r \uMCM_{\gr}(A)/(1)
\]
is essentially surjective \cite{Keller6} and hence
\[
\uMCM_{\gr}(A)\r \uMCM_{\gr}(\widehat{A})
\]
is also essentially surjective. In more complicated examples there is no
reason however why this should be the case. Nevertheless we have
the following result which is probably well-known.
\begin{proposition}
\label{ref-8.1-58}
Assume that $k$ has characteristic zero. Let
$A=k+A_1+A_2+\cdots$ be a left noetherian graded $k$-algebra. Put
$R=\widehat{A}$. Let $M\in \mod(R)$ be such that $\Ext^1_R(M,M)=0$. Then
$M$ is the completion of a finitely generated graded $A$-module $N$.
\end{proposition}
In the rest of this section we let the notations and hypotheses be as in the statement
of the proposition (in particular $k$ has characteristic zero). We denote
the maximal ideal of $R$ by $m$.

Let $E$ be the Euler derivation on $A$ and
$R$. I.e.  on $A$ we have $E(a)=(\deg a)a$ and we extend $E$ to $R$ in
the obvious way. If $M\in \mod(R)$ then we will define an Euler
connection as a $k$-linear map $\nabla:M\r M$ such that
$\nabla(am)=E(a)m+a\nabla(m)$. If $M=\widehat{N}$ for $N$ a graded
$A$-module then $M$ has an associated Euler connection  by extending
$\nabla(n)=(\deg n)n$ for $n$ a homogeneous element of $N$.
\begin{lemma}
 \label{ref-8.2-59} Let $M$ be a finitely generated $R$
 module. Then $M$ has an Euler connection if and only if $M$ is the
 completion of a finitely generated graded $A$-module.
\end{lemma}
\begin{proof} We have already explained the easy direction. Conversely
assume that $M$ has an Euler connection. For each $n$ we have that
$M/m^n M$ is finite dimensional and hence it decomposes into generalized
eigenspaces for $\nabla$.
\[
M/m^n M=\prod_{\alpha\in k}(M/m^n M)_\alpha  \qquad \text{(finite product)}
\]
Considering right exact sequences
\[
(m/m^2)^{\otimes n}\otimes  M/mM\r M/m^{n+1}M\r M/m^n M\r 0
\]
we easily deduce that the multiplicity of a fixed generalized eigenvalue
in $M/m^n M$ stabilizes a $n\r \infty$. Thus $M=\prod_{\alpha\in k} M_\alpha$
where $M_\alpha$ is a generalized eigenspace with eigenvalue $\alpha$.
We put $N'=\oplus_\alpha M_\alpha$. Then $N'$ is noetherian since
obviously any ascending chain of graded submodules of $N'$ can be
transformed into an ascending chain of submodules in $M$. If particular
$N'$ is finitely generated and we have $M=\widehat{N}'$.

Now $N'$ is $k$-graded and not $\ZZ$-graded. But we can decompose $N'$
along $\ZZ$-orbits and then by taking suitable shifts we obtain
a $\ZZ$-graded module with the same completion as $N'$.
\end{proof}
\begin{proof}[Proof or Proposition \ref{ref-8.1-58}]
Let $\epsilon^2=0$ and consider $M[\epsilon]$ where $A$ acts via
$a\cdot m=(a+E(a)\epsilon)m$.  We have a short exact sequence of $A$-modules
\[
0\r M\epsilon\r M[\epsilon]\r M\r 0
\]
which is split by hypotheses. Denote the splitting by $m+\nabla(m)\epsilon$.
For $a\in A$ we have
\[
am+\nabla(am)\epsilon=(a+E(a)\epsilon)(m+\nabla(m)\epsilon)
\]
and hence
\[
\nabla(am)=E(a)m+a\nabla(m)
\]
Hence $\nabla$ is an Euler connection and so we may invoke Lemma
\ref{ref-8.2-59} to show that $M=\widehat{N}$.
\end{proof}

\appendix

\section{Generators of singularity categories}

Throughout
$(A,\mf{m},k)$ is a (commutative) local noetherian ring, with maximal
ideal $\mf{m}$ and residue field $k$. The \emph{singularity category}
of $A$ is the Verdier quotient
\[
\dbsing(A) := \qderu{b}{\mod A}/\K^b(\proj A)
\]
of the bounded derived category of finitely generated $A$-modules by
the full subcategory of perfect complexes. Recall that a functor $F:
\cat{T} \lto \cat{S}$ is an \emph{equivalence up to direct summands}
if $F$ is fully faithful and every object $X \in \cat{S}$ is a direct
summand of $F(Y)$ for some $Y \in \cat{T}$. We say that $A$ is a
\emph{G-ring} if the canonical morphism from $A$ to its $\mf{m}$-adic
completion $A \lto \widehat{A}$ is regular \cite[\S 32]{Matsumura},
and that $A$ \emph{has an isolated singularity} if $A_\mf{p}$ is
regular for every non-maximal prime ideal $\mf{p}$ of $A$. Our main
result about singularity categories is the following:
\begin{proposition}\label{theorem:first_theorem} Let $A$ be a local
  noetherian ring with an isolated singularity, which is also a G-ring
  \textup{(}e.g. $A$ is essentially of finite type over a
  field\textup{)}. Then the canonical functor
\[
\gamma := - \otimes_A \widehat{A}: \dbsing(A) \lto \dbsing(\widehat{A})
\]
is an equivalence up to direct summands.
\end{proposition}

This is a special case of a general result by Orlov \cite{Orlov09} (which
was obtained indepedently). Our methods are quite different however. 

When $A$ is Gorenstein there is an equivalence, due to Buchweitz
\cite{BuchweitzAlone}, between $\dbsing(A)$ and the stable category of
maximal Cohen-Macaulay $A$-modules $\underline{\MCM}(A)$, so in this
case we obtain Proposition \ref{ref-1.5-5}. We remark that, in
general, $\gamma$ is not an equivalence (see e.g.\ Example
\ref{example:not_an_equivalence}).

Let us outline the proof of the proposition. Recall that a \emph{thick
  subcategory} of a triangulated category $\cat{T}$ is a triangulated
subcategory closed under retracts. Given an object $C$ of $\cat{T}$,
we say that an object $X$ is \emph{finitely built} from $C$ if it
belongs to the smallest thick subcategory of $\cat{T}$ containing
$C$. If every object of $\cat{T}$ has this property, that is, if there
are no proper thick subcategories of $\cat{T}$ containing $C$, then
$C$ is said to \emph{classically generate} $\cat{T}$.

The local ring $A$ and its completion $\widehat{A}$ have the same
residue field $k$, and it is not difficult to see that $\gamma$
induces an equivalence between the respective subcategories consisting
of objects finitely built from $k$. The subtlety lies in showing that,
because $A$ has an isolated singularity, \emph{every} object can be
finitely built from $k$. Our proof of this fact uses homotopy
colimits, which presents a technical problem since $\dbsing(A)$ lacks
infinite coproducts. One approach is to enlarge the category
$\dbsing(A)$ by considering the Verdier quotient
\[
\dbsinge(A) := \qderu{b}{\Mod A} / \K^b(\Proj A)
\]
of the bounded derived category of all $A$-modules by the full subcategory of bounded complexes of projective $A$-modules. By \cite[Proposition 1.13]{Orlov04} the canonical functor $\dbsing(A) \lto \dbsinge(A)$ is fully faithful, and $\dbsinge(A)$ turns out to contain enough coproducts (and thus homotopy colimits) for our purposes. Throughout $\qder{A}$ denotes the (unbounded) derived category of $A$-modules.

\medskip

The next proposition follows immediately from the work of Schoutens
\cite{Schoutens03}, but we give a direct proof in the special case of
an isolated singularity. The result also follows from the general
result by Orlov \cite{Orlov09} and Dyckerhoff
\cite{Dyckerhoff} has given a proof based on the theory of matrix
factorizations in the hypersurface case.


\begin{proposition}\label{theorem:dbsing_generated_k} A local noetherian ring $(A,\mf{m},k)$ has an isolated singularity if and only if $\dbsing(A)$ is classically generated by $k$.
\end{proposition}
\begin{proof} 
We begin with the easy direction. Suppose that $\dbsing(A)$ is classically generated by $k$, and let $\mf{p} \neq \mf{m}$ be a prime ideal. The canonical functor $- \otimes_A A_\mf{p}: \dbsing(A) \lto \dbsing(A_\mf{p})$ is identically zero, because it sends the generator $k$ to zero. The image of this functor contains the residue field $\kappa(\mf{p}) = A/\mf{p} \otimes_A A_\mf{p}$, from which we deduce that $\kappa(\mf{p})$ has finite projective dimension over $A_\mf{p}$. Hence $A_\mf{p}$ is regular, and we may conclude that $A$ has an isolated singularity.

Now suppose that $A$ has an isolated singularity, and let $M$ in $\qderu{b}{\mod A}$ be given. The idea is to write $M$ as a homotopy colimit\footnote{To be precise, we do not consider homotopy colimits in $\dbsinge(A)$, since coproducts in this category are rather subtle. Instead, we consider the image under the quotient functor $\qderu{b}{\Mod A} \lto \dbsinge(A)$ of homotopy colimits in $\qderu{b}{\Mod A}$.} of a sequence of bounded complexes with finite length cohomology; it follows that $M$ is a direct summand of one of the terms in this sequence, from which we conclude that $k$ classically generates. First, we set up some notation. Given $a \in A$, define complexes
\begin{align*}
\ukoszul{a}{A} &:= A \xlto{a} A, \text{ and } \koszul{a}{A} := A \xlto{\can} A[a^{-1}],
\end{align*}
both concentrated in degrees zero and one, and observe that the commutative diagram
\[
\xymatrix{
A \ar[d]_a \ar[r]^1 & A \ar[d]_{a^2} \ar[r]^1 & A \ar[d]_{a^3} \ar[r]^1 & \cdots\\
A \ar[r]_a & A \ar[r]_a & A \ar[r]_a & \cdots
}
\]
is a direct system of complexes $\ukoszul{a}{A} \lto \ukoszul{a^2}{A} \lto \ukoszul{a^3}{A} \lto \cdots$ with colimit~$\koszul{a}{A}$. More generally, given a sequence $\bold{a} = \{ a_1,\ldots,a_d \}$ in $A$, we define $\ukoszul{\bold{a}}{A} := \otimes_{j=1}^d \ukoszul{a_j}{A}$ and $\koszul{\bold{a}}{A} := \otimes_{j=1}^d \koszul{a_j}{A}$. Setting $\bold{a}^i = \{ a_1^i, \ldots, a_d^i \}$, there is a canonical isomorphism $\koszul{\bold{a}}{A} \cong \varinjlim_i \ukoszul{\bold{a}^i}{A}$ and thus a triangle
\begin{equation}\label{eq:dbsing_generated_k0}
\xymatrix@C+0.7pc{
\bigoplus_{i \ge 1} \ukoszul{\bold{a}^i}{A} \ar[r]^{1-\text{shift}} & \bigoplus_{i \ge 1} \ukoszul{\bold{a}^i}{A} \ar[r] & \koszul{\bold{a}}{A} \ar[r] &
}
\end{equation}
in the derived category $\qder{A}$. This triangle expresses the fact
that $\koszul{\bold{a}}{A}$ is the homotopy colimit of the
$\ukoszul{\bold{a}^i}{A}$ in $\qder{A}$. For background on homotopy
colimits, see \cite{BokNee93,NeemanBook}.

Now let $\bold{a}$ be a system of parameters for $A$, and extend the
augmentation morphism $\koszul{\bold{a}}{A} \xlto{\varepsilon} A$ to a
triangle $\koszul{\bold{a}}{A} \lto A \lto \cech{\bold{a}}{A} \lto$,
where the complex $\cech{\bold{a}}{A} := \Sigma \Ker(\varepsilon)$ is
given in each degree by $\cech{\bold{a}}{M}^t = \oplus_{i_0 < \cdots <
  i_t} A[a_{i_0}^{-1},\ldots,a_{i_t}^{-1}]$. Tensoring with $M$, we
obtain a triangle
\begin{equation}\label{eq:dbsing_generated_k1}
\koszul{\bold{a}}{A} \otimes_A M \lto M \lto \cech{\bold{a}}{A} \otimes_A M \lto
\end{equation}
in $\qder{A}$. Since $A$ has an isolated singularity,
$M[a_{i_0}^{-1}\cdots a_{i_t}^{-1}]$ has finite projective dimension
over $A[a_{i_0}^{-1}\cdots a_{i_t}^{-1}]$, and hence also over $A$,
for every sequence of indices $i_0 < \cdots < i_t$ in $\{ 1,\ldots,d
\}$. Here we use the fact that $A[a_{i_0}^{-1}\cdots a_{i_t}^{-1}]$
has finite projective dimension as an $A$-module\footnote{By induction
  this reduces to the observation that $\operatorname{pd}_A A[a^{-1}]
  \le 1$, which holds because $A[a^{-1}] = A[X]/(aX-1)$.}. We conclude
that $\cech{\bold{a}}{A} \otimes_A M$ is, up to isomorphism in
$\qder{A}$, a bounded complex of projective $A$-modules, whence the
triangle (\ref{eq:dbsing_generated_k1}) gives rise to an isomorphism
$\koszul{\bold{a}}{A} \otimes_A M \cong M$ in $\dbsinge(A)$. Note that
the coproduct $\oplus_{i \ge 1} \ukoszul{\bold{a}^i}{A} \otimes_A M$
is bounded, so tensoring (\ref{eq:dbsing_generated_k0}) with $M$
yields a triangle in $\dbsinge(A)$ of the form
\begin{equation}\label{eq:second_triangle}
  \xymatrix@C+0.7pc{
    \bigoplus_{i \ge 1} \ukoszul{\bold{a}^i}{A} \otimes_A M \ar[r]^{1-\operatorname{shift}} & \bigoplus_{i \ge 1} \ukoszul{\bold{a}^i}{A} \otimes_A M \ar[r] & M \ar[r] &.
  }
\end{equation}
In what follows, let $\Hom(-,-)$ denote morphism sets in
$\dbsinge(A)$. One can check (see Lemma
\ref{lemma:fg_commutes_coproducts} below) that $\Hom(M,-)$ commutes
with coproducts coming from $\qderu{b}{\Mod A}$ via the quotient
functor, so applying $\Hom(M,-)$ to (\ref{eq:second_triangle}) and
using the argument of \cite[Lemma 2.8]{Neeman96} we deduce that
\[
\Hom(M,M) \cong \varinjlim_i \Hom(M, \ukoszul{\bold{a}^i}{A} \otimes_A M).
\]
In particular, the identity $1_M: M \lto M$ corresponds to a split
monomorphism $M \lto \ukoszul{\bold{a}^k}{A} \otimes_A M$ in
$\dbsinge(A)$ for some $k \ge 1$. The functor $\dbsing(A) \lto
\dbsinge(A)$ is fully faithful, so $M$ is also a direct summand of
$\ukoszul{\bold{a}^k}{A} \otimes_A M$ in $\dbsing(A)$. The cohomology
modules of $\ukoszul{\bold{a}^k}{A} \otimes_A M$ have finite length
($\bold{a}$ is a system of parameters) so this complex is an iterated
extension in $\qderu{b}{\mod A}$ of finite direct sums of copies of
$k$. It is now clear that any thick subcategory of $\dbsing(A)$
containing $k$ must contain $M$, and since $M$ was arbitrary, this
completes the proof.
\end{proof}

\begin{lemma}\label{lemma:simple_factorisation} A morphism $\varphi: M
  \lto C$ in $\qder{A}$ with $M \in \qderu{b}{\mod A}$ and $C \in
  \K^b(\Proj A)$ factors, in $\qder{A}$, as $M \lto Q \lto C$ for some
  $Q \in \K^b(\proj A)$.
\end{lemma}
\begin{proof}
  We may, without loss of generality, assume that $M$ is a bounded
  above complex of finitely generated projective $A$-modules, that $C$
  is a bounded complex of free $A$-modules, and that $\varphi$ is a
  morphism of complexes. Let $n \in \mathbb{Z}$ be such that $C^i = 0$
  for $i < n$. The image of $\varphi^n: M^n \lto C^n$ is finitely
  generated, so let $Q^n$ be a finite free submodule of $C^n$ with the
  property that $\varphi^n$ factors as $M^n \lto Q^n \lto
  C^n$. Similarly, let $Q^{n+1}$ be a finite free submodule of
  $C^{n+1}$ with the property that $\operatorname{Im}(\varphi^{n+1})
  + \partial( Q^n ) \subseteq Q^{n+1}$, where $\partial$ is the
  differential. Then $\varphi^{n+1}$ factors as $M^{n+1} \lto Q^{n+1}
  \lto C^{n+1}$ and the differential restricts to a map $\partial|_Q:
  Q^n \lto Q^{n+1}$. Proceeding in this way we construct a bounded
  complex $Q$ of finite free $A$-modules and a factorization $M \lto Q
  \lto C$, as required.
\end{proof}

\begin{lemma}\label{lemma:fg_commutes_coproducts} Let $\{ X_i \}_{i
    \in I}$ be a family of bounded complexes of $A$-modules such that
  there exist $a,b \in \mathbb{Z}$ with $X_i^k = 0$ for all $k \notin
  [a,b]$ and $i \in I$. Then, given $M \in \qderu{b}{\mod A}$, the
  canonical map
\[
\oplus_i \Hom_{\dbsinge(A)}(M, X_i) \lto \Hom_{\dbsinge(A)}(M, \oplus_i X_i)
\]
is an isomorphism, where $\oplus_i X_i$ denotes the degree-wise
coproduct of complexes.
\end{lemma}
\begin{proof}
  By a standard argument, it is enough to prove that any morphism $M
  \lto \oplus_i X_i$ in $\dbsinge(A)$ factors through a finite
  subcoproduct. Such a morphism is defined by a roof
\begin{equation}\label{eq:fg_commutes_coproducts_roof}
\xymatrix@R-1pc{
& Y \ar[dl]_f \ar[dr]\\
M & & \oplus_i X_i
}
\end{equation}
in $\qderu{b}{\Mod A}$, where the cone of $f$ is a bounded complex
$C_f$ of projective $A$-modules. Extending $f$ to a triangle $Y \lto M
\lto C_f \lto$ in $\qderu{b}{\Mod A}$ we deduce from Lemma
\ref{lemma:simple_factorisation} that $M \lto C_f$ factors as $M \lto
Q \lto C_f$ for some $Q \in \K^b(\proj A)$. Let $C'$ denote the cone
of $M \lto Q$. From the octahedral axiom applied to the pair $(M \lto
Q, Q \lto C_f)$ we obtain a commutative diagram in $\qderu{b}{\Mod A}$
of the form
\[
\xymatrix{
\Sigma^{-1} C' \ar[d]_h \ar[dr]\\
Y \ar[r]_f & M
}
\]
where the cone of $h$ belongs to $\K^b(\Proj A)$. The upshot is that
the morphism in $\dbsinge(A)$ represented by the roof in
(\ref{eq:fg_commutes_coproducts_roof}) may also be represented by a
roof with $Y \in \qderu{b}{\mod A}$ (replace $Y$ with $\Sigma^{-1}
C'$). In this case $Y$ is compact in $\qderu{b}{\Mod A}$ by
\cite[Proposition 6.15]{Rouquier}, so the morphism $Y \lto \oplus_i
X_i$ in the roof factors through a finite subcoproduct, which implies
that $M \lto \oplus_i X_i$ factors through a finite subcoproduct in
$\dbsinge(A)$.
\end{proof}

\begin{proof}[Proof of Proposition \ref{theorem:first_theorem}] To begin with, let $A$ denote an arbitrary local noetherian ring, and consider the canonical functor
\[
\gamma':= - \otimes_A \widehat{A}: \dbsinge(A) \lto \dbsinge(\widehat{A}).
\]
Restriction of scalars defines a functor $(-)_A: \qderu{b}{\Mod
  \widehat{A}} \lto \qderu{b}{\Mod A}$ which sends a bounded complex
of projective $\widehat{A}$-modules to a bounded complex of flat
$A$-modules. Since flat $A$-modules have finite projective dimension
by \cite[Part II, Corollary 3.2.7]{Raynaud71}, there is an induced
functor
\[
(-)_A: \dbsinge(\widehat{A}) \lto \dbsinge(A)
\]
right adjoint to $\gamma'$. The unit of this adjunction is the canonical morphism
\[
1 \lto ( - \otimes_A \widehat{A})_A,
\]
which is obviously an isomorphism on $k$, and thus also an isomorphism
on the smallest thick subcategory $\cat{S}$ of $\dbsinge(A)$
containing $k$. By a standard argument of category theory, the
restriction of $\gamma'$ to $\cat{S}$ is fully faithful. In
particular, $\gamma$ induces an equivalence of the smallest
triangulated subcategory of $\dbsing(A)$ containing $k$ with the
smallest triangulated subcategory of $\dbsing(\widehat{A})$ containing
$k$.

Now we assume that $A$ is a G-ring with an isolated singularity. The
(only) reason for assuming that $A$ is a G-ring is that this
guarantees that the completion $\widehat{A}$ has an isolated
singularity \cite[Lemma 2.7]{Wiegand98}. By Proposition
\ref{theorem:dbsing_generated_k} the subcategory $\cat{S}$ includes
the image of $\dbsing(A)$ under the canonical embedding $\dbsing(A)
\lto \dbsinge(A)$, from which we infer that $\gamma$ is fully
faithful. It follows from a second application of Proposition
\ref{theorem:dbsing_generated_k} that the thick closure of
$\dbsing(A)$ in $\dbsing(\widehat{A})$ is all of
$\dbsing(\widehat{A})$. Since the thick closure of a triangulated
subcategory is just the class of all direct summands of objects in the
subcategory \cite[Remark 2.1.39]{NeemanBook}, $\gamma$ is an
equivalence up to direct summands.
\end{proof}
It is easy to construct examples where $\gamma$ is not an equivalence.
It suffices to give a Cohen-Macauly module over the completion of a Gorenstein
local ring $\hat{A}$ which is not \emph{extended} from $A$, i.e.\ which is
not of the form $\hat{M}$ for a Cohen-Macaulay $A$-module.
\begin{example}\label{example:not_an_equivalence} Let $A =
  \mathbb{C}[X,Y]_{(X,Y)}/(X^3 + X^2 - Y^2)$ be the local ring of a
  node, so the completion of $A$ is isomorphic to the reduced ring $S
  = \mathbb{C}[[U,V]]/(UV)$. This is a singularity of type $(A_1)$ and
  by \cite[(9.9)]{Yoshino90} there are up, to isomorphism, exactly
  three indecomposable maximal Cohen-Macaulay $S$-modules, which are
\[
S, \; \mf{p} = US, \text{ and } \mf{q} = VS.
\]
Clearly $S/\mf{p} \cong \mf{q}$, whence $\mf{q} \cong \Sigma \mf{p}$
in $\dbsing(S)$. Since $\mf{p},\mf{q}$ are minimal prime ideals,
$S_\mf{p}$ and $S_\mf{q}$ are fields, and it follows from a result of
Levy and Odenthal \cite[Theorem 6.2]{Levy96} that a finitely generated
$S$-module $M$ is extended if and only if
$\operatorname{rank}_{S_\mf{p}}(M_\mf{p}) =
\operatorname{rank}_{S_\mf{q}}(M_\mf{q})$. Hence $\mf{p}$ and $\mf{q}$
are not extended, and thus not in the essential image of $\gamma$, but
their direct sum $\mf{p} \oplus \mf{q}$ is extended. This corresponds
to the fact that the nodal curve is irreducible, while the curve $XY =
0$ has two irreducible components. Another argument that $\mf{p}
\oplus \mf{q} \cong \mf{p} \oplus \Sigma \mf{p}$ in $\dbsing(S)$
belongs to the essential image of $\gamma$ uses K-theory: simply apply
\cite[Corollary 4.5.12]{NeemanBook}.

Note that $\{ U - V \}$ is a system of parameters for $S$. It follows
from the proof of Proposition \ref{theorem:dbsing_generated_k} that
$\mf{p}$ is a direct summand in $\dbsing(S)$ of $\ukoszul{(U-V)^n}{S}
\otimes \mf{p}$ for some $n \ge 1$. In fact, $\ukoszul{(U-V)^n}{S}
\otimes \mf{p} = \mf{p} \xlto{U^n} \mf{p}$ is quasi-isomorphic to
$\Sigma^{-1} \mf{p}/\mf{p}^{n+1}$, and $\mf{p}$ is a direct summand of
$\Sigma^{-1} \mf{p}/\mf{p}^2$ in $\dbsing(S)$. To see this, observe
that there is a triangle in the derived category
\[
\mf{p} \xlto{U} \mf{p} \lto \mf{p}/\mf{p}^2 \lto \Sigma \mf{p},
\]
and $U: \mf{p} \lto \mf{p}$ is zero in $\dbsing(S)$ (as it factors via $S$) so we may conclude that $\mf{p} \oplus \Sigma \mf{p} \cong \mf{p}/\mf{p}^2$ in $\dbsing(S)$. Since $\mf{p}/\mf{p}^2$ is isomorphic as an $S$-module to the residue field $\mathbb{C}$, we see for a third time that $\mf{p} \oplus \Sigma \mf{p} \cong \mathbb{C}$ is in the essential image of $\gamma$.
\end{example}

\begin{remark}\label{remark:henselization} Denoting by $A^h$ the
  Henselization of $A$, the ring homomorphisms $A \lto A^h \lto
  \widehat{A}$ give rise to a factorization of $\gamma$ as the
  composite
\[
\dbsing(A) \xlto{\gamma_1} \dbsing(A^h) \xlto{\gamma_2} \dbsing(\widehat{A}),
\]
where $\gamma_1 = - \otimes_A A^h$ and $\gamma_2 = - \otimes_{A^h}
\widehat{A}$. In the situation of Proposition \ref{theorem:first_theorem},
$\gamma_2$ is an equivalence: up to a shift, every object of
$\dbsing(\widehat{A})$ is a finitely generated module $M$ free on the
punctured spectrum, and by Elkik's theorem \cite[Th\'eor\`eme
3]{Elkik74} such modules can be descended to the Henselization; that
is, there exists a finitely generated $A^h$-module $N$ such that $M
\cong \widehat{N}$. In particular, $\gamma$ is an honest equivalence
(not just up to direct summands) when $A$ is Henselian.
\end{remark}

Now we give the proof of
Proposition~\ref{ref-1.4-3}.  In \cite{Krause05} Krause produces an
embedding $\mu: \dbsing(A) \hookrightarrow \K_{\ac}(\Inj A)$, where
$\K_{\ac}(\Inj A)$ is the homotopy category of $C_{\ac}(\Inj A)$ of
acyclic complexes of injective $A$-modules. This category is compactly
generated, and $\mu$ induces an equivalence up to direct summands
between $\dbsing(A)$ and the full subcategory of compact objects in
$\K_{\ac}(\Inj A)$.

The embedding $\mu$ produces a DG-enhancement for $\dbsing(A)$ where for
$M,N\in \dbsing(A)$ we put
\[
\RHom_{\dbsing(A)}(M,N)=\underline{\Hom}_{C_{\ac}(\Inj A)}(\mu(M),\mu(N))
\]
If $A$ is a noetherian $\ZZ$-graded ring (not necessarily commutative)
then we may define the graded singularity category
$\dbsing^{\text{gr}}(A)$ in the obvious way.

Since $\dbsing^{\text{gr}}(A)$ has an analogous DG-enhancement as
$\dbsing(A)$ we may define the orbit category
$\dbsing^{\text{gr}}(A)/(1)$ (see \cite{Keller6}). By construction
$\dbsing^{\text{gr}}(A)/(1)$ is a triangulated category (with a
DG-enhancement) equipped with an exact functor
\[
\sigma:\dbsing^{\text{gr}}(A)\r \dbsing^{\text{gr}}(A)/(1)
\]
such that $\dbsing^{\text{gr}}(A)/(1)$ is classically generated by its
essential image and such  that for $M,N\in \dbsing^{\text{gr}}(A)$ we 
have
\[
\Hom_{\dbsing^{\text{gr}}(A)/(1)}(\sigma M,\sigma N)=\bigoplus_i \Hom_{\dbsing^{\text{gr}}(A)}(M,N(i))
\]
Forgetting the grading yields an exact functor
\[
F:\dbsing^{\text{gr}}(A)\r \dbsing(A)
\]
which makes the shift $(1)$ isomorphic to the identity functor. 
Hence by the universal property of orbit categories $F$ factors canonically through
\[
\tilde{F}:\dbsing^{\text{gr}}(A)/(1)\r \dbsing(A)
\]
\begin{lemma} \label{fff} The functor $\tilde{F}$ is fully faithful. 
\end{lemma}
\begin{proof} We have to prove that for $M,N \in \dbsing^{\text{gr}}(A)$ we
have 
\[
\Hom_{\dbsing(A)}(M,N)=\bigoplus_i \Hom_{\dbsing^{\text{gr}}(A)}(M,N(i))
\]
By considering cones over suitable truncated projective resolutions we
may assume that $M,N$ are finitely generated $A$-modules.

We then use the well-known formula
\[
\Hom_{\dbsing(A)}(M,N)=\dirlim_n \Hom_{D(A)}(\Omega^n M,\Omega^n N)
\]
and the corresponding formula in the graded case. This reduces us
to proving 
\[
\Hom_{D(A)}(\Omega^n M,\Omega^n N)=\bigoplus_i \Hom_{D^{\text{gr}}(A)}(\Omega^n M,\Omega^n N(i))
\]
This follows easily by replacing $M$ by a projective resolution. 
\end{proof}
\begin{proposition}
\label{gradedcase} Let $A=k+A_1+A_2\cdots$
 be a finitely generated commutative graded $k$-algebra with the
augmentation ideal $m=A_{>0}$ defining an isolated singularity. Then
we have equivalences 
\[
\dbsing^{\text{gr}}(A)/(1)\xrightarrow{\widetilde{F}} \dbsing(A)\xrightarrow{(-)_m}
\dbsing(A_m)\xrightarrow{\widehat{A}\otimes_{A}-} \dbsing(\hat{A})
\]
\end{proposition}
\begin{proof} The third functor is an equivalence because of
Proposition \ref{theorem:first_theorem}. The second functor is
an equivalence because of \cite{Orlov04}. Finally in 
Lemma \ref{fff} we have shown that $\tilde{F}$ is fully faithful.
So we have to show that it is essentially surjective. 
This is clear by Proposition \ref{theorem:dbsing_generated_k} since $k$ lies in the essential image of $\tilde{F}$. 
\end{proof}
Again we obtain Proposition \ref{ref-1.4-3} by invoking Buchweitz's equivalence
$\dbsing(A)\cong \underline{\text{MCM}}(A)$. 


\begin{thebibliography}{10}

\bibitem{AZ}
M.~Artin and J.~J. Zhang, {\em Noncommutative projective schemes}, Adv. in
 Math. {\bf 109} (1994), no.~2, 228--287.

\bibitem{ARP}
M.~Auslander, M.~Platzeck, and I.~Reiten, {\em Coxeter functors without
 diagrams}, Trans. Amer. Math. Soc. {\bf 250} (1979), 1--46.

\bibitem{A3}
M.~Auslander and I.~Reiten, {\em Graded modules and their completions}, Topics
 in Algebra, Banach Center Publications {\bf 2} (1990), 181--192.

\bibitem{Baeth07}
N.~Baeth, \emph{A {K}rull-{S}chmidt theorem for one-dimensional rings
  of finite {C}ohen-{M}acaulay type}, J. Pure Appl. Algebra \textbf{208}
  (2007), no.~3, 923--940.

\bibitem{BeligiannisContra}
A.~Beligiannis, \emph{The homological theory of contravariantly finite
  subcategories: {A}uslander-{B}uchweitz contexts, {G}orenstein categories and
  (co-)stabilization}, Comm. Algebra \textbf{28} (2000), no.~10, 4547--4596.

\bibitem{BokNee93}
M.~B{\"o}kstedt and A.~Neeman, \emph{Homotopy limits in triangulated
  categories}, Compositio Math. \textbf{86} (1993), no.~2, 209--234.

\bibitem{Bondal}
A.~Bondal and A.~Polishchuk, {\em Homological properties of associative
 algebras: the method of helices}, Russian Acad. Sci. Izv. Math {\bf 42}
 (1994), 219--260.

\bibitem{BondalVdB}
A.~Bondal and M.~Van~den Bergh, {\em Generators and representability of
 functors in commutative and noncommutative geometry}, Mosc. Math. J. {\bf 3}
 (2003), no.~1, 1--36, 258.

\bibitem{BuchweitzAlone}
R.-O.~Buchweitz, {\em Maximal {C}ohen-{M}acaulay modules and {T}ate cohomology
 over {G}orenstein rings}, unpublished manuscript, 155 pages, 1987.

\bibitem{Buchweitz}
R.-O.~Buchweitz, D.~Eisenbud, and J.~Herzog, {\em Cohen-{M}acaulay modules on
 quadrics}, Singularities, representation of algebras, and vector bundles
 (Lambrecht, 1985) (Berlin), Springer, Berlin, 1987, pp.~58--116.

\bibitem{Chen07}
X.-W.~Chen, \emph{Relative singularity categories and Gorenstein-projective
  modules}, arXiv:0709.1762v1.

\bibitem{Dyckerhoff} T.~Dyckerhoff, {\em Compact generators in categories of matrix factorizations}, arXiv:0904.4713v3.

\bibitem{Elkik74}
R.~Elkik, \emph{Solution d'\'equations au-dessus d'anneaux
  hens\'eliens}, Quelques probl\`emes de modules ({S}\'em. {G}\'eom. {A}nal.,
  \'{E}cole {N}orm. {S}up., {P}aris, 1971-1972), Soc. Math. France, Paris,
  1974, pp.~116--132. Ast\'erisque, No. 16.

\bibitem{Frankild08}
A.~Frankild, S.~Sather-Wagstaff, R.~Wiegand, \emph{Ascent of module structure, vanishing of Ext, and extended modules}, Michigan Math. J. \textbf{57} (2008), 321--337.

\bibitem{Gabriel2}
P.~Gabriel, {\em Auslander-{R}eiten sequences and representation-finite
 algebras}, Representation theory, I (Proc. Workshop, Carleton Univ., Ottawa,
 Ont., 1979) (Berlin), Lecture Notes in Math., vol. 831, Springer, Berlin,
 pp.~1--71.

\bibitem{EGAIV}
A.~Grothendieck and J.~Dieudonn{\'e}, {\em {\'E}tude locale de sch{\'e}mas et
 des morphismes de sch{\'e}mas}, Inst. Hautes {\'E}tudes Sci. Publ. Math. {\bf
 20, 24, 28, 32} (1964-1967).

\bibitem{IY} O.~Iyama and Y.~Yoshino, {\em Mutation in triangulated
   categories and rigid {C}ohen-{M}acaulay modules},
  Invent. Math.  {\bf 172}  (2008), 117--168. 

\bibitem{Keller1}
B.~Keller, {\em Deriving {DG}-categories}, Ann. Sci. {\'E}cole Norm. Sup. (4)
 {\bf 27} (1994), 63--102.

\bibitem{Keller10}
\bysame, {\em On the construction of triangle equivalences}, Derived
 equivalences for group rings (Berlin), Lecture Notes in Math., vol. 1685,
 Springer, Berlin, 1998, pp.~155--176.

\bibitem{Keller6}
\bysame, {\em On triangulated orbit categories}, Doc. Math. {\bf 10} (2005),
 551--581 (electronic).

\bibitem{KeRe}
B.~Keller and I.~Reiten, {\em Acyclic {C}alabi-{Y}au categories (with an appendix by Michel Van den Bergh)},
 math.RT/0610594, to appear in Compositio.

\bibitem{Krause05}
H.~Krause, \emph{The stable derived category of a {N}oetherian scheme},
  Compos. Math. \textbf{141} (2005), no.~5, 1128--1162.

\bibitem{Levy96}
L.~Levy, C.~Odenthal, \emph{Krull-Schmidt theorems in dimension $1$}, Trans. Amer. Math. Soc. \textbf{348} (1996) 3391--3455.

\bibitem{Matsumura}
H.~Matsumura, \emph{Commutative ring theory}, Cambridge Univ. Press, Cambridge, 1989.

\bibitem{Neeman3}
A.~Neeman, {\em The connection between the ${K}$-theory localization theorem of
 {T}homason, {T}robaugh and {Y}ao and the smashing subcategories of
 {B}ousfield and {R}avenel}, Ann. Sci. \'Ecole Norm. Sup. (4) {\bf 25} (1992),
 no.~5, 547--566.

\bibitem{Neeman96}
A.~Neeman, \emph{The {G}rothendieck duality theorem via {B}ousfield's
  techniques and {B}rown representability}, J. Amer. Math. Soc. \textbf{9}
  (1996), no.~1, 205--236.

\bibitem{NeemanBook}
\bysame, \emph{Triangulated categories}, Annals of Mathematics Studies, vol.
  148, Princeton University Press, Princeton, NJ, 2001.

\bibitem{Orlov2}
D.~Orlov, {\em Derived categories of coherent sheaves and triangulated
 categories of singularities}, arXiv:math/0503632.

\bibitem{Orlov}
\bysame, {\em Projective bundles, monoidal transformations and derived functors
 of coherent sheaves}, Russian Acad. Sci. Izv. Math {\bf 41} (1993), no.~1,
 133--141.

\bibitem{Orlov04}
\bysame, \emph{Triangulated categories of singularities and {D}-branes in
  {L}andau-{G}inzburg models}, Tr. Mat. Inst. Steklova \textbf{246} (2004),
  no.~Algebr. Geom. Metody, Svyazi i Prilozh., 240--262.

\bibitem{Orlov09}
\bysame, {\em Formal completions and idempotent completions of triangulated categories of singularities},	arXiv:0901.1859v1.

\bibitem{popescu}
D.~Popescu, {\em General {N}\'eron desingularization and approximation}, Nagoya
 Math. J. {\bf 104} (1986), 85--115.

\bibitem{Raynaud71}
M.~Raynaud and L.~Gruson, \emph{Crit\`eres de platitude et de
  projectivit\'e. {T}echniques de ``platification'' d'un module}, Invent. Math.
  \textbf{13} (1971), 1--89.

\bibitem{Ri}
J.~Rickard, {\em Morita theory for derived categories}, J. London Math. Soc.
 (2) {\bf 39} (1989), 436--456.

\bibitem{Rouquier}
R.~Rouquier, {\em Dimensions of triangulated categories}, arXiv:math/0310134.

\bibitem{Schoutens03}
H.~Schoutens, \emph{Projective dimension and the singular locus}, Comm. Algebra \textbf{31} (2003), no.~1, 217--239.

\bibitem{Se}
J.~P. Serre, {\em Faisceaux alg{\'e}briques coh{\'e}rents}, Ann. of Math. (2)
 {\bf 61} (1955), 197--278.

\bibitem{Wiegand98}
R.~Wiegand, \emph{Local rings of finite {C}ohen-{M}acaulay type}, J. Algebra
  \textbf{203} (1998), no.~1, 156--168.

\bibitem{Yoshino90}
Y.~Yoshino, \emph{Cohen-{M}acaulay modules over {C}ohen-{M}acaulay rings},
  London Mathematical Society Lecture Note Series, vol. 146, Cambridge
  University Press, Cambridge, 1990.

\end{thebibliography}
\def\cprime{$'$} \def\cprime{$'$} \def\cprime{$'$}
\ifx\undefined\bysame
\newcommand{\bysame}{\leavevmode\hbox to3em{\hrulefill}\,}
\fi

\end{document}